\documentclass{amsart}
\usepackage{enumitem}
\usepackage{hyperref}
\newtheorem{theo}{Theorem}
\newtheorem{ques}[theo]{Question}

\newtheorem{coro}[theo]{Corollary}

\newtheorem{Rema}[theo]{Remark}
\theoremstyle{remark}

\newcommand{\Z}{\mathbb{Z}}

\newcommand{\SL}{\mathrm{SL}}

\newcommand{\Sp}{\mathrm{Sp}}
\newcommand{\hol}{\operatorname{hol}}

\begin{document}
% -------------------------------------------------------------------------
\title{A remark on $\mathbb{Z}^d$-covers of Veech surfaces}
\author{Angel Pardo}
%\email{angel.pardo.j@gmail.com}

% -------------------------------------------------------------------------
% Abstract
\begin{abstract}
In this note we are interested in the dynamics of the linear flow on infinite periodic $\mathbb{Z}^d$-covers of Veech surfaces.
An elementary remark allows us to show that the kernel of some natural representations of the Veech group acting on homology is ``big''. In particular, the same is true for the Veech group of the infinite surface, answering a question of Pascal Hubert.
We give some applications to the dynamics on wind-tree models where the underlying compact translation surface is a Veech surface.
\end{abstract}

% -------------------------------------------------------------------------
\maketitle

\section{Introduction}

Let $X$ be a Veech surface, that is, a translation surface whose Veech group $\Gamma$ is a lattice, or more generally, let us just assume that $\Gamma$ is non-elementary.
Let $\Sigma$ be the finite set of singularity points (and, possibly, marked points) of $X$.
Since the intersection form $\langle\cdot,\cdot\rangle$ is non-degenerate between $H^1(S\setminus\Sigma,\Z)$ and $H^1(S,\Sigma,\Z)$, every connected $\Z^d$-cover is defined by a $d$-tuple of primitive, linearly independent elements $\mathbf{f}=(f_1,\dots,f_d)$ in the group of relative cohomology $H^1(S,\Sigma,\Z)$. 

For simplicity, we restrict ourselves to the case when $\mathbf{f}$ is an absolute covector, that is, it is a $d$-tuple of independent elements in the group of \emph{absolute} cohomology $H^1(S,\Z)$, and denote by $X_\mathbf{f}$ the corresponding $\Z^d$-periodic (connected) translation surface. (In this case, $X_\mathbf{f}$ is actually a surface.)

\subsection*{Veech group representation}
For $g\in\Gamma$, the Kontsevich--Zorich cocycle (or the Torelli map for the corresponding affine automorphisms) defines a symplectic map on $H^1(X)$ which preserves $H^1(X,\Z)$. Thus, it defines a representation $\rho_{H^1}$ of $\Gamma$ on the symplectic group $\Sp(H^1(X,\Z))$.

If $F$ is a subspace of $H^1(X)$ which is invariant under this action, the restriction to $F$ gives another representation $\rho_{F}: \Gamma \to \SL(F)$, which, in general, is not faithful. Furthermore, this representation is neither symplectic nor defined over $\Z$ in general. However, if the subspace $F$ is symplectic or defined over $\Z$, so the representation $\rho_{F}$ is.

(For the sake of having an invariant subspace sometimes we could want to get rid of finite order elements $\Gamma$. Since this (virtually) does not make any difference, we always assume that $\Gamma$ has not finite order (elliptic) elements.)
%In particular, $-id\notin \Gamma$ and hence, every representation $\rho_F$ descends to a representation of $\PSL(X)$ on $\PSL(F)$.)

\subsection*{Integer invariant subspaces}
We can always assume that for $\mathbf{f}=(f_1,\dots,f_d)$ there are integer invariant subspaces $F^{(j)} \subset H^1(X)$, irreducible over $\Z$, such that $f_i\in F=\oplus_j F^{(j)}$, $i=1,\dots,d$. Note that, in general, one could have $F^{(j)}=H^1(X)$.

However, if one imposes the zero-drift condition, that is, that $\hol(f_i) = 0$, then $F\subset \ker\hol$, which is a codimension $2$ subspace in $H^1(X)$.
The zero-drift condition is a necessary condition for recurrence of the linear flow on $X_\mathbf{f}$ on almost every direction and it is also sufficient in the particular case of $d=1$.
Thus, it is a natural condition to consider when interested in dynamics of the linear flow on $X_\mathbf{f}$ and from now on, we assume that the zero drift condition holds, that is, $\hol(f_i) = 0$, $i=1,\dots,d$; and we fix $F^{(j)},F\subset \ker\hol$, the integer invariant subspaces as above.

%Note that the zero-drift condition imposes $d \leq \dim\spa(\ker\hol \cap H^1(X,\Z))\in\{0,\dots,2g-2\}$.

The most important objects in this note are the subgroups $K_j=\ker \rho_{F^{(j)}}<\Gamma$.
Since $K_j$ and $\Gamma$ are Fuchsian groups and $K_j\lhd\Gamma$, their limit sets coincide unless $K_j$ is trivial (this is a result of Matsuzaki--Taniguchi~1998; see, e.g., \cite[Lemma 5.4]{HoW}). %[find the exact result on reference [14: Matsuzaki--Taniguchi~1998] of \cite{HoW}]).
Based on Thurston ideas, Hooper--Weiss~\cite[Theorem 5.5]{HoW} proved that when $\Gamma$ is a lattice and $F^{(j)}\subset \ker\hol$ is a $2$-dimensional integer subspace, the representation $\rho_{F^{(j)}}$ is never faithful and thus, $K_j$ non-trivial.
To our knowledge, it is not known whether a representation given by a higher dimensional integer invariant space could be faithful or not. We do not aim to treat this deeper problem here.
However, after this note, this question becomes central to the study of $\Z^d$-covers of Veech surfaces, see Question~\ref{ques:question} below.

We are interested in the case when all $K_j$'s are non-trivial. 
More precisely, we are interested in the group $K=\ker \rho_F=\cap_j K_j$.
Observe that if $F$ contains the tautological plane (generated by the Poincar\'e dual of $\Re\omega$ and $\Im\omega$), the representation $\rho_F$ is always faithful and thus $K$ trivial. This is another reason to assume zero-drift.
%[Maybe, say something when this is not the case, and the representation is faithful]

\section{An elementary remark}
The main tool in this note is given by the following elementary remark.

\begin{Rema}\label{rema:commutators}
%Let $\psi_i:H\to H_i$ be group homomorphisms and denote by $k_i$ the kernel of $\psi_i$, $i=1,2$. Then $[k_1,k_2]\lhd k_1\cap k_2$.
Let $A,B$ be two normal subgroups of $G$. Then $[A,B]\subset A\cap B$.
\end{Rema}
%\begin{proof}[Proof (to be deleted)]
%Let $a\in k_1$ and $b\in k_2$. Then
%\[ \psi_1([a,b])=\psi_1(aba^{-1}b^{-1})=\psi_1(a)\psi_1(b)\psi_1(a)^{-1}\psi_1(b)^{-1}=\psi_1(b)\psi_1(b)^{-1}=1_H.\]
%Let $a\in \ker\phi$ and $b\in \ker\varphi$. Then
%\[ \psi([a,b])=\psi(aba^{-1}b^{-1})=\psi(a)\psi(b)\psi(a)^{-1}\psi(b)^{-1}=\psi(b)\psi(b)^{-1}=id.\]
%\end{proof}

And the following concequence.

\begin{theo}\label{theo:theorem}
Suppose that $\Gamma$ is of the first-kind (resp. second-kind) and that all $K_j$'s are non-trivial. Then $K=\ker \rho_F=\cap_{j=1}^n K_j$ is of the first-kind (resp. second-kind).
In particular, the Veech group of $X_\mathbf f$ is of the first-kind (resp. second-kind).
\end{theo}
\begin{proof}
First of all, note that the last conclusion follows from the former since $K$ is a subgroup of the Veech group of $X_\mathbf f$. % (\cite[Lemma 5.4]{HoW}).

Since $K_j$ is a non-trivial normal subgroup of $\Gamma$, the limit set of both groups coincide (\cite[Lemma 5.4]{HoW}).
If $n=1$, we are done. Suppose $n\geq 2$, let $H_1=K_1$ and $H_j=\left\langle[H_{j-1},K_j]\right\rangle_\mathrm{N}$ for $j=2,\dots,n$, where $\left\langle\cdot\right\rangle_\mathrm{N}$ denotes the normal closure.
By Remark~\ref{rema:commutators}, it follows that $H_n\lhd K$. It suffices then to show that $H_n$ is non-trivial, since $K$ is a normal subgroup of $\Gamma$ and in such case, the limit set of both groups would coincide.

By induction:
\begin{itemize}[leftmargin=*]
\item By definition, $H_1=K_1 \lhd \Gamma$ and by hypothesis, it is non-trivial.
\item As before, $K_n$ is a non-trivial normal subgroup of $\Gamma$, and by the induction hypothesis, the same is true for $H_{n-1}$. Then both (as any non-trivial normal subgroup of $\Gamma$) are of the first kind.
It follows that $H_n=\left\langle[H_{n-1},K_n]\right\rangle_\mathrm{N}$ is non-trivial. 
In fact, take any infinite order element $h\in H_{n-1}$ and any infinite order element $k\in K_n$ whose fixed points in the boundary do not coincide (note that this is possible as soon as one of the groups is non-elementary or both are, but they belong to different one-parameter subgroups). It follows that $h$ and $k$ do not commute, that is $id\neq [h,k] \in H_n$.\qedhere
\end{itemize}
\end{proof}

Combining this with \cite[Theorem 5.6]{HoW}, the following is a direct consequence.

\begin{coro} \label{coro:2d}
Suppose that $X$ is a Veech surface %, that is, that $\Gamma$ is a lattice group
and that all the $F^{(j)}$ are $2$-dimensional. Then $K=\cap_j K_j$ is of the first-kind.
In particular, the Veech group of $X_\mathbf f$ is of the first-kind.
\end{coro}

The previous result relies on Thurston ideas and does not extend straightforwardly to the higher dimensional case.
Thus, by Theorem~\ref{theo:theorem}, the following question is central to the study of $\Z^d$-covers of Veech surfaces.

\begin{ques}\label{ques:question}
Under which conditions on a irreducible $\Gamma$-invariant integer subspace $F\subset H^1(X)$, the group $K=\ker\rho_F$ is non-trivial?
Does the zero-drift condition, that is, $F\subset \ker\hol$, suffices?
\end{ques}

\section{Applications to wind-tree models}
Periodic wind-tree models ---both the classical model and the Delecroix--Zorich variant--- yields $\Z^2$-periodic translation surfaces defined by cocycles lying in $2$-dimensional subspaces (see, e.g., \cite{DHL,DZ,P}).
It follows, by Corollary~\ref{coro:2d}, that when the underlying compact surface is a Veech surface, the Veech group of the infinite surface is of the first kind.

\subsection{Diffusion rates}
Since, for non-elementary groups, the limit set coincides with the closure of the set of fixed points of hyperbolic elements, it follows that there is a dense set of directions in which the polynomial diffusion rate is zero.
Thus, the ideas of Crovisier--Hubert--Lanneau~\cite{CHL} can be applied to show the following.

\begin{theo}
Let $\Pi$ be a wind-tree model (including the Delecroix--Zorich variant) such that the underlying compact translation surface is a Veech surface. Then, the set of polynomial diffusion rates is the whole interval $[0,1[$.
\end{theo}

\begin{proof}
By \cite[Theorem~1]{DHL} and \cite[Theorem~1.5]{P}, for every wind-tree model $\Pi$, the polynomial diffusion rate in a generic direction is constant and positive, say $\delta_\Pi>0$.
By the work of Crovisier--Hubert--Lanneau~\cite{CHL}, the set of polynomial diffusion rates  $\Delta_\Pi\subset [0,1[$ is connected, and $[\delta,1[{}\subset\Delta_\Pi$. But, by Corollary~\ref{coro:2d}, there are hyperbolic elements in the kernel of the corresponding representation. These give zero polynomial diffusion rates, that is, $0\in\Delta_\Pi$. Finally, since $\Delta_\Pi$ is connected, $[0,1[{}\subset\Delta_\Pi$.
\end{proof}

\subsection{Ergodicity}
Using the ergodicity criterium of Hubert--Weiss~\cite[Theorem~1]{HuW}, which extends naturally to $\Z^d$-covers, it is possible to show that even if the ergodicity is rare, by a result of Fr\c{a}czek--Ulcigrai~\cite[Corollary~1.3]{FU}, in the measure theoretic sense, it is typical in the topological sense, that is, there is a $G_\delta$-dense set of directions for which the billiard flow is ergodic in the wind-tree model.

\begin{theo}
Let $\Pi$ be a wind-tree model (including the Delecroix--Zorich variant) such that the underlying compact translation surface is a Veech surface. Then, there is a $G_\delta$-dense set of directions for which the billiard flow is ergodic in $\Pi$.
\end{theo}

\begin{proof}
Since, by Corollary~\ref{coro:2d}, the Veech group is of the first kind, by~\cite[Corollary~16]{HuW}, it is enough to find two strips (in the infinite translation surface) such that the intersection of the corresponding core curves $\gamma$ (in the compact translation surface) with the cocycles $f_1,f_2$ defining the $\Z^2$-cover generates a rank-$2$ lattice in $\Z^2$.
But the horizontal and vertical strips do the job. Thus, there is a $G_\delta$-dense set of directions for which the billiard flow is ergodic in $\Pi$.
\end{proof}

%-------------------------------------------------------------------------
% Bibliography

%-------------------------------------------------------------------------
\end{document}